\documentclass[11pt, a4paper]{article}

\usepackage{geometry}                		
\usepackage{amsmath, amsfonts, amssymb, amsthm, bm}
\usepackage{graphicx, subfigure, tikz, booktabs}
\usepackage{epstopdf}
\usepackage{cleveref}
\ifpdf
  \DeclareGraphicsExtensions{.eps,.pdf,.png,.jpg}
\else
  \DeclareGraphicsExtensions{.eps}
\fi

\newtheorem{theorem}{Theorem}
\newtheorem{problem}{Problem}
\newtheorem{definition}{Definition}

\begin{document}

\title{Identifying an acoustic source in a two-layered medium from multi-frequency phased or phaseless far-field patterns}

\author{Yan Chang\thanks{School of Mathematics, Harbin Institute of Technology, Harbin, P. R. China. ({21B312002@stu.hit.edu.cn}).},
            Yukun Guo\thanks{School of Mathematics, Harbin Institute of Technology, Harbin, P. R. China. ({ykguo@hit.edu.cn}).},
	  and Yue Zhao\thanks{School of Mathematics and Statistics, and Key Lab NAA–MOE, Central China Normal University, Wuhan, P. R. China. ({zhaoyueccnu@163.com}).}}
	  
\date{}

\maketitle

\begin{abstract}
This paper presents a method for reconstructing an acoustic source located in a two-layered medium from multi-frequency phased or phaseless far-field patterns measured on the upper hemisphere.  The interface between the two media is assumed to be flat and infinite, while the source is buried in the lower half-space. In the phased case, a Fourier method is proposed to identify the source based on far-field measurements.  This method assumes that the source is compactly supported and can be represented by a sum of Fourier basis functions.  By utilizing the far-field patterns at different frequencies, the Fourier coefficients of the source can be determined, allowing for its reconstruction. For the case where phase information is unavailable, a phase retrieval formula is developed to retrieve the phase information.  This formula exploits the fact that the far-field patterns are related to the source through a linear operator that preserves phase information. By developing a suitable phase retrieval algorithm, the phase information can be recovered. Once the phase is retrieved, the Fourier method can be adopted to recover the source function. Numerical experiments in two and three dimensions are conducted to validate the performance of the proposed methods.
\end{abstract}



\section{Introduction}
 
 The field of inverse scattering problems, which gains momentum from significant scientific and industrial applications, has witnessed remarkable growth in recent years \cite{Bao2015, CK18, CK19}. 
Being an essential topic in the field of inverse problems, the research background for the layered medium is rooted in the observation that numerous natural and artificial materials possess a two-layered structure.  For example, the skin has a dermis and an epidermis \cite{gongcheng1}, and printed circuit boards have a copper layer and a dielectric layer \cite{gongcheng2}.  In addition, many technical processes involve using two-layered media, such as high-speed printing \cite{gongcheng3}, high-precision cutting, and thin-film solar cells.  The exploration of two-layered media is crucial for comprehending their physical attributes and enhancing their performance across various applications. However, compared with single-layered media, the study of two-layered media is more challenging due to the intricate interactions between the two layers.  Furthermore, reflection, refraction, and transmission of waves occur at these interfaces, which complicates the analysis and computation of the wave propagation in a layered medium.

The research background for the layered medium is extensive and intertwines multiple disciplines, including materials science, physics, chemistry, and engineering. Mathematically, the study of two-layered media has garnered enduring attention. The authors of \cite{ZX, ZX2} have established increasing stability results for the inverse source problem (ISP) in the one-dimensional Helmholtz equation in a multi-layered medium. In \cite{HXYZ23}, the authors have developed increasing stability estimates for the inverse source scattering problem in two-dimensional space and introduced a recursive Kaczmarz-Landweber iteration scheme with incomplete data. To identify point sources from multi-frequency sparse far-field patterns, a multi-step numerical scheme is proposed in \cite{LS}. For a deeper understanding of the direct and inverse scattering problems in a two-layered medium, we refer to \cite{Liu17, Liu6, LLLL, LZ10, YL18} and the associated reference materials.

Although the ISP in two-layered media has been the focus of research interest thus far, most existing methods are either sampling-based (direct but qualitative) or iterative-type (quantitative but time-consuming). For the sake of rapid reconstruction with high precision, it is necessary to investigate novel imaging algorithms.

In this work, we aim to reconstruct the source function with compact support in a two-layered medium by utilizing the multi-frequency far-field data collected on the upper half-space.
The problem under consideration is more sophisticated than that in an unbounded space. A reason accounting for this is that the medium consists of two distinct layers, each with its own specific properties, and the propagation of waves through this medium is influenced by both layers.  Additionally, the availability of far-field models is limited to the upper half space, forming a major challenge due to the practical physical model. Moreover, another obstruction to solving the problem lies in the fact that the observation aperture is closely related to the refractive index (defined as the ratio of the speed above and below the interface). Thus, when the refractive index decreases, the aperture shrinks sharply. It is well known that the small observation aperture will significantly increase the ill-posedness. All the aforementioned difficulties constitute the challenges of the underlying problem.

Motivated by the Fourier method proposed for the ISP in the full space for the Helmholtz equation \cite{Wang2023, Wang2017} and the biharmonic equation \cite{CGYZ24}, we shall investigate the feasibility of extending this method to the two-layer media case. A notable distinction to previous full-space problems is that the far-field measurements in our two-layered model are only accessible in the upper half-space. A computational advantage of the Fourier method is that we can directly determine the Fourier coefficients from the far-field data without prior knowledge of the source function. Therefore, it is crucial to establish a correspondence between the admissible wavenumbers and the Fourier coefficients. Nevertheless, we highlight that the distinct sound speeds above and below the interface result in different wavenumbers, which may render the existing Fourier method inapplicable. Hence, extending the full-space approach to its two-layer version is non-trivial, thus novel modifications deserve investigation. 

In many practical applications, only the intensity information of the radiated field can be measured, whereas the phase information is difficult to access or completely unavailable. 
To address this issue, we introduce an easy-to-implement phase retrieval technique by incorporating auxiliary known reference source points into the inverse source system. This technique allows us to recover the phase information, thereby reformulating the phaseless problem into the standard ISP with phase information.
Compared to the full-space case, a notable difference is that the configuration of reference source points plays an important role in the specific details of phase recovery. In a nutshell, although the deployment of the auxiliary point in the lower space offers computational convenience, it comes at the expense of practical engineering costs. In comparison, it is more convenient and realistic to place the auxiliary point sources in the upper half-space, but this naturally requires an immense endeavor to derive novel computational schemes. 

Once the phase information is retrieved, the phaseless inverse source problem (PLISP) can be easily reduced to its phased version, which can be solved using the Fourier method introduced here. Instead of delving into mathematical intricacies such as the stability and solvability of the equation system in the phase-retrieval process, we shall focus on feasible implementations and demonstrate its effectiveness through extensive numerical results conducted in two and three dimensions.

The remaining part of the work is organized as follows: In the next section, we present the forward and inverse source model with a two-layer media. The Fourier method for solving the phased problem is given in \Cref{sec: Fourier}. Then, the phase retrieval algorithm for the phaseless inverse source problem is discussed in \Cref{sec: PR}. Next, several numerical examples are provided in \Cref{sec: example} to illustrate the effectiveness of the proposed methods, and the conclusions follow in \Cref{sec: conclusion}.

\section{Problem setup}\label{sec: problem_setup}

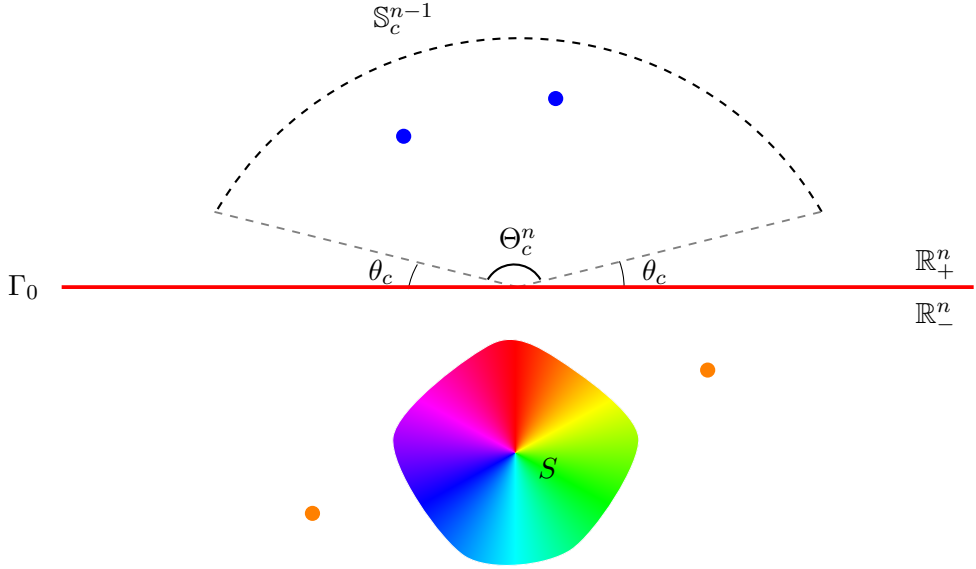
\begin{figure}[h]
	\centering
	\usepgflibrary{shadings}
	\begin{tikzpicture}[radius=1cm, delta angle=20]
		\draw[dashed, black,thick] (4,3) arc(30:150:{4.6}); 
		\draw[black,thick] (.3,2.1) arc(30:150:{.4}); 
		\draw node at(-6.5,2) {$\Gamma_0$};
		\draw node at (1.8,2.2) {$\theta_c$};
		\draw node at (-1.8,2.2) {$\theta_c$};
		\draw (.8,2) +(0:.6cm) arc [start angle=0];
		\draw (-0.8,2) +(150:.6cm) arc [start angle=150];
		\draw node at (0,2.6) {$\Theta_c^n$};
		\draw node at (-1.5,5.56) {$\mathbb{S}_c^{n-1}$};
		\draw[dashed, gray, line width=0.25mm] (0,2) -- (4,3); 
		\draw[dashed, gray, line width=0.25mm] (0,2) -- (-4,3); 
		\pgfmathsetseed{47}			
		\filldraw[shading=color wheel,draw=white] plot [smooth cycle, samples=6, domain={1:6}] (\x*360/6-4*rnd:-2.1cm+2cm*rnd);
		\draw node at (.4, -.4) {$S$};
		\fill [blue] (.5,4.5) {circle (0.1)};
		\fill [blue] (-1.5,4) circle (0.1);
		\fill [orange] (-2.7,-1) circle (0.1);
		\fill [orange] (2.5,.9) circle (0.1);
		\draw node at (5.5,1.6) {$\mathbb{R}_-^n$};
		\draw node at (5.5,2.3) {$\mathbb{R}_+^n$};
		\draw [red, line width=0.5mm] (-6,2) -- (6,2); 
	\end{tikzpicture}
	\caption{The source function $S$ is located in the lower half-space.}
		\label{fig:Setup}
\end{figure}

Let us start this section with a mathematical description of the model setup. 
As shown in \Cref{fig:Setup}, the flat plane $\Gamma_0:=\{x\in\mathbb{R}^n:x_n=0\}$ delimits the whole space $\mathbb{R}^n\,(n=2,3)$ into two half-spaces: the upper half space $\mathbb{R}_+^n:=\{x=(x_1, x_2, \cdots, x_n)\in\mathbb{R}^n: x_n>0\}$ and the lower half space $\mathbb{R}_-^n:=\{x\in\mathbb{R}^n: x_n<0\}.$  The unbounded domains $\mathbb{R}_\pm^n$ are filled with homogeneous and isotropic acoustic material with wave speed $c_\pm>0$, respectively. 

Given source function $S$, the radiated field $u$ is governed by the Helmholtz equation
\begin{equation}\label{eq: Helmholtz}
	\Delta u+k^2(x)u=S,\quad x\in\mathbb{R}^n,
\end{equation}
with $k(x)=k_\pm=\omega/c_\pm$ for $x\in\mathbb{R}^n_\pm$ and $\omega>0$ is the wave frequency.  Throughout this article, we assume that $S$ is independent of $k_\pm$ and located in $\mathbb{R}_-^n,$ with $\text{supp}\,S\subset V_0.$ Here, $V_0\subset\mathbb{R}_-^n, n=2,3$ is some square ($n=2$) or cuboidal ($n=3$) domain. Under the assumption that $\Gamma_0$ is flat, the radiated field $u$ is supposed to satisfy the following Sommerfeld radiation condition
\begin{equation}\label{eq: src}
	\lim\limits_{\lvert x \rvert \to\infty}\lvert x \rvert^{\frac{n-1}{2}}\left(
	\frac{\partial u}{\partial\lvert x \rvert}-\mathrm{i}k_\pm u
	\right)=0.
\end{equation}
We want to point out that, for a general geometry setting on $\Gamma_0$, the Sommerfeld radiation condition \eqref{eq: src} may no longer be valid to describe the asymptotic behavior of the radiated wave $u$ as $\lvert x \rvert\to\infty$ (see \cite{YL18}).

In \Cref{fig:Setup}, the observation aperture $\Theta_c^n$ is defined by
$$
	\Theta_c^n:=
	\begin{cases}
		(\theta_c,\pi-\theta_c), & n=2,\\
		\left(\theta_c,{\pi}/{2}\right], & n=3,
	\end{cases}
$$
with the critical angle given by
$$
	\theta_c=
	\begin{cases}
		\arccos\left(\dfrac{c_+}{c_-}\right), &\text{if}\  c_->c_+,\\
		0, &\text{if}\  c_-\le c_+.
	\end{cases}
$$
Correspondingly, the observation directions of interest are defined through:
\begin{align*}
	\mathbb{S}_c^1 & :=\left\{\hat{x}\in\mathbb{S}^1:\hat{x}=(\cos\theta,\sin\theta),\ \theta\in\Theta_c^2\right\}, \\
	\mathbb{S}_c^2 & :=\left\{\hat{x}\in\mathbb{S}^2:\hat{x}=(\cos\phi\cos\theta,\sin\phi\cos\theta,\sin\theta),\ \theta\in\Theta_c^3,\ \phi\in[0,2\pi)\right\},
\end{align*}
 respectively.
For later use, the transmitted direction $\hat{x}^t\in\mathbb{S}^{n-1}\cap\mathbb{R}_+^n$ for $\hat{x}\in\mathbb{S}_c^n$ is denoted as  
\begin{align*}
	\hat{x}^t:=\left\{
	\begin{aligned}
		&\left(\frac{c_-}{c_+}\cos\theta,\sqrt{1-\frac{c_-^2}{c_+^2}\cos^2\theta}\right),&&n=2,\\
		&\left(\frac{c_-}{c_+}\cos\theta\cos\phi,\frac{c_-}{c_+}\cos\theta\sin\phi,\sqrt{1-\frac{c_-^2}{c_+^2}\cos^2\theta}\right),&&n=3.
	\end{aligned}
	\right.
\end{align*}
For $x=(x_1, x_2, \cdots, x_n)\in\mathbb{R}^n(n=2,3),$   we also define $\tilde{x}=(x_1, x_2, \cdots, x_{n-1})\in\mathbb{R}^{n-1}$ and $x^s:=(x_1, x_2, \cdots, -x_n)\in\mathbb{R}^n$. 

It is well known that the solution to \eqref{eq: Helmholtz}--\eqref{eq: src} is given by
\begin{align}u(x, \omega)=-\int_{V_0}\Phi_\omega(x, y)S(y)\mathrm{d}y,\end{align}
where $\Phi_\omega(x,y)$ is the outgoing Green's function to the following transmission problem:
\begin{align}\label{eq: Greeneq}\left\{
	\begin{aligned}
		&\Delta_x\Phi_\omega(x, y)+k(x)^2\Phi_\omega(x,y)=-\delta_y(x), \quad x\in\mathbb{R}_+^n\cup\mathbb{R}_-^n,\\
		&\left[\Phi_\omega(x, y)\right]=\left[\partial_{x_n}\Phi_\omega(x, y)\right]=0, \quad x\in\Gamma_0.
	\end{aligned}
		\right.
\end{align}
with, in each domain, the Rellich-Sommerfeld radiation condition satisfied:
\begin{align*}
	\lim\limits_{r=\lvert x \rvert\to\infty}r^{\frac{n-1}{2}}\left(\partial_r{\Phi}_\omega-\mathrm{i}k_\pm \Phi_\omega\right)=0.
\end{align*}
In \eqref{eq: Greeneq}, $[\,\cdot\,]$ denotes the jump across the interface $\Gamma_0,$ and $\delta_y(x):=\delta(x-y)$ is the Dirac function in $\mathbb{R}^n.$
As $\Gamma_0$ considered here is a flat plane, the Green function has the following explicit expressions \cite{LS, PA17, YL18}:
\begin{align}
	\Phi_\omega(x,y)=\left\{
	\begin{aligned}
		&\frac{1}{2\pi}\int_{\mathbb{R}^{n-1}}\frac{\mathrm{e}^{-\gamma_1x_n+\gamma_2y_n}}{\gamma_1+\gamma_2}\mathrm{e}^{\mathrm{i}\xi\cdot(\tilde{x}-\tilde{y})}\mathrm{d}\xi,&&y\in\mathbb{R}_-^n,\\
		&\frac{1}{2\pi}\int_{\mathbb{R}^{n-1}}
		\frac{\mathrm{e}^{-\gamma_1\lvert x_n-y_n\rvert}+\gamma\mathrm{e}^{-\gamma_1(x_n+y_n)}}{2\gamma_1}
		\mathrm{e}^{\mathrm{i}\xi\cdot(\tilde{x}-\tilde{y})}
		\mathrm{d}\xi,&&y\in\mathbb{R}_-^n,
	\end{aligned}
	\right.
\end{align}
where $\gamma_i:=\left(\lvert\xi\rvert^2-k_i^2\right)^{1/2}$ has non-negative real part and non-positive imaginary part, and
 $\gamma:=(\gamma_1-\gamma_2)/(\gamma_1+\gamma_2)$.
Further, the layered Green's function has the following asymptotic behavior \cite{LS, PA17, YL18}:
\begin{equation}\label{eq: far}
	\Phi_\omega(x, y)=\beta_n\frac{\mathrm{e}^{\mathrm{i}k_+\lvert x \rvert}}{\lvert x \rvert^{\frac{n-1}{2}}}\left\{\Phi_\omega^\infty(\hat{x},y)+o(1)\right\},
\quad\text{as}\quad \vert x\rvert\to\infty,
\end{equation}
which holds uniformly for all directions $\hat{x}\in\mathbb{S}_c^{n-1},$ with
$$
	\beta_n=\frac{\mathrm{e}^{\mathrm{i}\pi/4}}{\sqrt{8k_+\pi}}\left(\mathrm{e}^{-\mathrm{i}\frac{\pi}{4}}\sqrt{\frac{k_+}{2\pi}}\right)^{n-2},
$$
and 
\begin{align}\label{eq:phifar}
\Phi_\omega^\infty(\hat{x},y)=\left\{
\begin{aligned}
	&T(\theta)\mathrm{e}^{-\mathrm{i}k_-\hat{x}^t\cdot y},&&y\in\mathbb{R}_-^n,\\
	&H(\theta)\mathrm{e}^{-\mathrm{i}k_+\hat{x}\cdot y^s}+\mathrm{e}^{-\mathrm{i}k_+\hat{x}\cdot y},&&y\in\mathbb{R}_+^n,
\end{aligned}
\right.
\quad\hat{x}\in\mathbb{S}_c^{n-1}.
\end{align}
Here,
\begin{align*}
	T(\theta)&=\frac{2\sin\theta}{\sin\theta+\sqrt{c_+^2/c_-^2-\cos^2\theta}},\quad\theta\in\Theta_c^n,\\
	H(\theta)&=\frac{\sin\theta-\sqrt{c_+^2/c_-^2-\cos^2\theta}}{\sin\theta+\sqrt{c_+^2/c_-^2-\cos^2\theta}},\quad\theta\in\Theta_c^n.
\end{align*}
Accordingly, for $x\in\mathbb R_+^n,$ $u(x, \omega)$ admits the following asymptotic expansion:
\begin{align}
	u(x, \omega)=\beta_n\frac{\mathrm{e}^{\mathrm{i}k_+\lvert x \rvert}}{\lvert x \rvert^{\frac{n-1}{2}}}\left\{u^\infty(\hat{x}, \omega)+o(1)\right\},
	\quad\text{as}\ \vert x\rvert\to\infty,
\end{align}
where $u^\infty(\hat{x}, \omega)$ defined on the upper half unit sphere/circle $\mathbb S_+^{n-1}$ is known as the far-field pattern (scattering amplitude) with $\hat{x}$ being the observation direction:
\begin{align}\label{eq: uinf}
	u^\infty(\hat{x}, \omega)=T(\theta)\int_{V_0}\mathrm{e}^{-\mathrm{i}k_-\hat{x}^t\cdot y}S(y)\mathrm{d}y.
\end{align}

Let $\Omega_N=\{\omega_j\}_{j=1}^N,\,(N\in\mathbb{N})$ be an admissible set consisting of a finite number of frequencies, and we are now to propose the two inverse problems under consideration as follows:
\begin{problem}[multi-frequency ISP with far-field]\label{problem1}
	Determine the source function $S(x)$ from the multi-frequency far-field data $\{u^\infty(\hat{x}, \omega):\hat{x}\in\mathbb{S}_+^{n-1},\, \omega\in\Omega_N\}.$ 
\end{problem}
\begin{problem}[multi-frequency PLISP]\label{problem2}
	Recover $S(x)$ from the multi-frequency phaseless far-field data $\{\lvert u^\infty(\hat{x}, \omega)\rvert:\hat{x}\in\mathbb{S}_+^{n-1},\, \omega\in\Omega_N\}.$ 
\end{problem}

To tackle \Cref{problem2}, we adopt a two-stage strategy. Initially, we introduce several appropriate auxiliary source points into the inverse system to recover the phase information. Subsequently, we transform \Cref{problem2} into \Cref{problem1}.

Before heading for the inverse problem, we introduce several notations and the relevant Sobolev spaces in the rest of this section. Without loss of generality, we take $a>0,\,L>0$ such that the cubic domain
$$
V_0=\left(-\frac a 2,\frac a 2\right)^{n-1}\times\left(-L,0\right),
$$
contains the support of the source function $S,$ i.e., 
$
\text{supp }S\subset V_0.
$ 
For $\sigma>0,$ the periodic Sobolev space $H^\sigma(V_0)$ is defined by
$$_{}
H^\sigma(V_0)=\{v\in L^2(V_0):\|v\|_\sigma<\infty\},
$$
equipped with the norm
$$
\|v\|_\sigma=\left(\sum_{\vec{l}\in\mathbb{Z}^n}\left(1+{\lvert{\vec{l}}\rvert}^2\right)\lvert\hat{v}_{\vec{l}}\,\rvert^2\right)^{1/2},
$$
where 
	$$
	\hat{v}_{\vec{l}}=\frac{1}{a^n}\int_{V_0}v(x)\overline{\phi_{\vec{l}}\,(x)}\mathrm{d}x
	$$
is the Fourier coefficient and the Fourier basis function $\phi_{\vec{l}}\,(x)$ is given by
$$
\phi_{\vec{l}}\,(x)=\mathrm{e}^{\mathrm{i}\frac{2\pi}{a}\vec{l}\cdot x},\quad\vec{l}=(l_1,l_2,\cdots,l_n)\in\mathbb{Z}^n.
$$
Within this preparation, the main idea of this article is to approximate the source function $S$ by
\begin{align}\label{eq: SN}
	S_N(x)=\sum_{\lvert\vec{l}\rvert_\infty\le N}\hat{s}_{\vec l}\,\phi_{\vec{l}}\,(x),
\end{align}
where $\hat{s}_{\vec l}$ is the Fourier coefficient defined by
\begin{align}\label{eq: Fourier}
\hat{s}_{\vec{l}}=\frac{1}{a^n}\int_{V_0}S(x)\overline{\phi_{\vec{l}}\,(x)}\mathrm{d}x.
\end{align}
For $S_N$ defined by \eqref{eq: SN}, we have the following approximation result (\cite{IP15}):
\begin{theorem}
Let $S$ be a function in $H^\sigma(V_0)\,(\sigma>0),$ then the following estimate holds:
$$
\|S-S_N\|_\mu\le N^{\mu-\sigma}\|S\|_\sigma,\quad0\le\mu\le\sigma.
$$	
\end{theorem}

\section{Fourier method for ISP} \label{sec: Fourier}
In this section, we shall develop a Fourier method to deal with \Cref{problem1}. To appropriately utilize the multi-frequencies data, we are supposed to introduce the admissible frequencies.
\begin{definition}[Admissible wave numbers in the lower half-space]\label{def:admissible}
	Let $\lambda>0$ be a sufficiently small constant and 
	\begin{align}
		\vec{l}_0:=\left\{
		\begin{aligned}
			&(\lambda,0),&&n=2,\\
			&(\lambda,0,0),&&n=3.\\
		\end{aligned}
		\right.
	\end{align}
Let
\begin{align}
	\hat{x}^t=\hat{x}_{\vec{l}}=\left\{
	\begin{aligned}
		&\frac{\vec{l}}{\left|\vec{l}\,\right|},\quad\vec{l}\in\mathbb{Z}^n\backslash\{\vec{0}\},\\
		&\frac{\vec{l}_0}{\left|\vec{l}_0\right|},\quad\vec{l}=\vec{0},
	\end{aligned}
	\right.
\end{align}
then the admissible transmitted directions are defined by
\begin{equation}\label{eq: XN}
\mathcal{X}_N=
\left\{
\hat{x}_{\vec{l}}:\theta_{\vec{l}}\in\Theta_c^n,\ \lvert\vec{l}\rvert_\infty\le N
\right\},
\end{equation}
with 
$$
\tan\theta_{\vec{l}}=\left\{
\begin{aligned}
&\frac{l_2}{l_1},&&	n=2,\\
&\frac{l_3}{\sqrt{l_1^2+l_2^2}},&&n=3.
\end{aligned}
\right.
$$
The corresponding admissible wave number in the lower half-space  $\mathbb{R}_-^n $ are defined by
\begin{align}
	\mathbb{K}_{-,N}=\left\{{k_{-,{\vec{l}}}}:\theta_{\vec{l}}\in\Theta_c^n,\,1\le\lvert{\vec{l}}\rvert_\infty\le N\right\}\cup\{k_{-,{\vec{0}}}\}
\end{align}
with
\begin{align}
	{k_{-,{\vec{l}}}}:=\left\{
		\begin{aligned}
		&\frac{2\pi}{a}\lvert\vec{l}\rvert,&& \text{if}\quad\theta_{\vec{l}}\in\Theta_c^n,\,1\le\lvert{\vec{l}}\rvert_\infty\le N,\\
		&\frac{2\pi}{a}\lambda,&& \text{if}\quad{\vec{l}}={\vec 0}.
	\end{aligned}
	\right.
\end{align}
\end{definition}

Moreover, the admissible frequencies set $\Omega_N$, the admissible wave number set $\mathbb{K}_{+,N}$ in the upper half-space $\mathbb{R}_+^n$, and the observation angles  are uniquely determined, i.e.,
$$
	\Omega_N=c_-\mathbb{K}_{-,N},\quad\mathbb{K}_{+,N}=\Omega_N/c_+,\quad\frac{c_-}{c_+}\cos\theta=\frac{l_1}{\lvert{l}\rvert}.
$$

We now deliver the following uniqueness theorem:
\begin{theorem}
	Under \Cref{def:admissible}, the Fourier coefficients $\{\hat{s}_{\vec{l}}\}$ of $S$ in \Cref{eq: SN} can be uniquely determined by the far-field pattern $\{u^\infty(\hat{x}_{\vec{l}}, \omega_{\vec{l}}):\,\theta_{\vec{l}}\in\Theta_c^n\}.$
\end{theorem}
\begin{proof}
	Let $S$ be the exact source that produces the far field data $\{u^\infty(\hat{x}_{\vec{l}}, \omega_{\vec{l}}):\,\theta_{\vec{l}}\in\Theta_c^n\}.$
	 From the definition of the Fourier coefficients \eqref{eq: Fourier}, we derive that for $\vec{l}\in\mathbb{Z}^n\backslash\left\{\vec{0}\right\},$ 
	 \begin{align*}
	 	\hat{s}_{\vec{l}}&=\frac{1}{a^n}\int_{V_0}S(y)\overline{\phi_{\vec{l}}\,(y)}\mathrm{d}y\\
	 	&=\frac{1}{a^n}\int_{V_0}S(y)\exp\left(-\mathrm{i}\frac{2\pi}{a}\vec{l}\cdot y\right)\mathrm{d}y\\
	 	&=\frac{1}{a^n}\int_{V_0}S(y)\exp\left(-\mathrm{i}\frac{2\pi}{a}\left|\vec{l}\,\right|\frac{\vec{l}\cdot y}{\left|\vec{l}\,\right|}\right)\mathrm{d}y\\
	 	&=\frac{1}{a^n}\int_{V_0}S(y)\exp\left(-\mathrm{i}k_{\vec{l}}\,\hat{x}_{\vec{l}}\cdot y \right)\mathrm{d}y.
	 \end{align*}
 This combined with \eqref{eq: uinf} give rise to
 \begin{align}\label{eq: sl}
 	\hat{s}_{\vec{l}}=\frac{u^\infty(\hat{x}_{\vec{l}}, \omega_{\vec{l}})}{a^nT(\theta)},
 \end{align}
provided that the transmitted direction $\hat{x}^t$ is taken to be $\hat{x}_{\vec{l}}$.

For the case ${\vec{l}}=\vec{0},$ we derive from \eqref{eq: uinf} that
\begin{align*}
	\frac{u^\infty(\hat{x}, \omega_{\vec{0}})}{a^nT(\theta)}&=
	\frac{1}{a^n}\int_{V_0}S(y)\overline{\phi_{\vec{l}_0}(y)}\mathrm{d}y\\
	&=
	\frac{1}{a^n}\int_{V_0}\left(\hat{s}_{\vec{0}}+\sum_{\vec{l}\in\mathbb{Z}^n\backslash\{\vec{0}\}}\hat{s}_{\vec{l}}\,\phi_{\vec{l}}\,(y)\right)\overline{\phi_{\vec{l}_0}(y)}\mathrm{d}y\\
	&=\frac{\hat{s}_{\vec{0}}}{a^n}\int_{V_0}\overline{\phi_{\vec{l}_0}(y)}\mathrm{d}y+\frac{1}{a^n}\sum_{\vec{l}\in\mathbb{Z}^n\backslash\{\vec{0}\}}\hat{s}_{\vec{l}}
	\int_{V_0}\phi_{\vec{l}}\,(y)\overline{\phi_{\vec{l}_0}(y)}\mathrm{d}y\\
	&=\frac{\sin(\lambda\pi)}{\lambda\pi}\hat{s}_{\vec{0}}+\frac{1}{a^n}\sum_{\vec{l}\in\mathbb{Z}^n\backslash\{\vec{0}\}}\hat{s}_{\vec{l}}
	\int_{V_0}\phi_{\vec{l}}\,(y)\overline{\phi_{\vec{l}_0}(y)}\mathrm{d}y,
\end{align*}
which further gives the following formula to compute $\hat{s}_{\vec{0}}$ after truncating the infinite series by a finite order $N\in\mathbb{N}$, explicitly,
\begin{equation}\label{eq: s0}
    \hat{s}_{\vec{0}}\approx\frac{\lambda\pi}{a^n\sin(\lambda\pi)}\left(
    \frac{u^\infty(\hat{x}, \omega_{\vec{0}})}{T(\theta)}-\sum_{1\le\lvert\vec{l}\rvert_\infty\le N}\hat{s}_{\vec{l}}
    \int_{V_0}\phi_{\vec{l}}\,(y)\overline{\phi_{\vec{l}_0}\,(y)}\mathrm{d}y
    \right).
\end{equation}
\end{proof}

Combining \eqref{eq: SN}, \eqref{eq: sl}, and \eqref{eq: s0} gives the truncated Fourier series $S_N$ of the form
\begin{align}
	S_N(x)=\hat{s}_{\vec{0}}+\sum_{1\le\lvert\vec{l}\rvert_\infty\le N}\hat{s}_{\vec{l}}\,\phi_{\vec{l}}\,(x),
\end{align}
which can be viewed as an approximation to the source function. 

In the following, instead of investigating the mathematical justification with regard to the uniqueness and stability of the proposed scheme, we shall head for an intriguing numerical exploration, and further consider the case where the measured data is phaseless.

\section{Phase retrieval for PLISP}\label{sec: PR}

The strategy to reconstruct the source from the phaseless data (\Cref{problem2}) can be divided into two steps: In the first step, we retrieve the phase information by adding several artificial source points to the setup, which transform the PLISP into its phased problem; In the second step, the Fourier method developed in Section \ref{sec: Fourier} can be utilized to produce the final source reconstruction. We focus our concentration on the first step in this section.

Motivated by the phase retrieval technique developed in \cite{ZGLL18, Wang2023, CGZ24}, we shall propose a novel formula to retrieve the phase information by introducing appropriate auxiliary source points to the model setup. Noticing \eqref{eq:phifar}, the locations of the auxiliary artificial source points would have an important influence on this process. 

We now present the method to recover the phase information. For each observation direction $\hat{x}\in\mathbb{S}^{n-1},\,n=2,3,$ we take two points $z_j:=\alpha_j\hat{x}$ with $j=1,2,\,\alpha_j\in\mathbb{R}.$ Then the far-field pattern of the fundamental solution at points $z_j$ is given by
\begin{align}\label{eq:far_zj}
	\Phi_\omega^\infty(\hat{x},z_j)=\left\{
	\begin{aligned}
		&T(\theta)\mathrm{e}^{-\mathrm{i}k_-\hat{x}^t\cdot z_j},&&z_j\in\mathbb{R}_-^n,\\
		&H(\theta)\mathrm{e}^{-\mathrm{i}k_+\hat{x}\cdot z_j^s}+\mathrm{e}^{-\mathrm{i}k_+\hat{x}\cdot z_j},&&z_j\in\mathbb{R}_+^n,
	\end{aligned}
	\right.
	\quad\hat{x}\in\mathbb{S}_c^{n-1}.
\end{align}
By defining the scaling parameters
\begin{align}
	c_j=\frac{\|u^\infty(\cdot, \omega)\|_\infty}{\|\Phi_\omega^\infty(\cdot,z_j)\|_\infty},\quad j=1,2,
\end{align}
with $\|\cdot\|_\infty=\|\cdot\|_{L^\infty(\mathbb{S}^{n-1})}$, the auxiliary function $\phi_j(x, \omega)=-c_j\Phi_\omega(x, z_j)$ satisfies the following inhomogeneous Helmholtz equation
$$
\Delta\phi_j(x;\omega)+k(x)^2\phi_j(x)=c_j\delta_{z_j}(x)\quad\text{in }\ \mathbb{R}^n_+\cup\mathbb{R}^n_-.
$$
By denoting $v_j=u+\phi_j\,(j=1,2)$, we derive that $v_j$ is the unique solution to 
$$
	\begin{cases}
		\left(\Delta +k(x)^2\right)v_j(x, \omega)=\left(S+c_j\right)(x),\ \ {\rm in}\ \mathbb{R}^n,\\
		\lim\limits_{r=\lvert x \rvert\to\infty}r^{\frac{n-1}{2}}\left(\partial_r{v_j}(x, \omega)-\mathrm{i}k_\pm v_j(x, \omega)\right)=0.
	\end{cases}
$$
Furthermore, the far-field pattern of $v_j$ is given by
\begin{align}\nonumber
v_{j}^\infty & =u^\infty+\phi_j^\infty\\
\nonumber & =u^\infty-c_j\Phi_\omega^\infty\\
\label{eq: vj_far} & =u^\infty-c_j
\begin{cases}
	T(\theta)\mathrm{e}^{-\mathrm{i}k_-\hat{x}^t\cdot z_j},&z_j\in\mathbb{R}_-^n,\\
	H(\theta)\mathrm{e}^{-\mathrm{i}k_+\hat{x}\cdot z_j^s}+\mathrm{e}^{-\mathrm{i}k_+\hat{x}\cdot z_j},&z_j\in\mathbb{R}_+^n,
\end{cases}
\quad j=1,2.
\end{align}

We now turn to the following phase retrieval problem:
\begin{problem}[Phase retrieval]
	Let $v_j^\infty$ be the far-field pattern corresponding to the radiated field $v_j,\,j=1,2.$
	Reconstruct the phased radiated data $\{u^\infty(\hat{x};\omega):\hat{x}\in\mathbb{S}_+^{n-1},\omega\in\Omega_N\}$ from the following phaseless data
	\begin{align*}
		&\{\vert u^\infty(\hat{x}, \omega)\rvert:\hat{x}\in\mathbb{S}_+^{n-1},\ \omega\in\Omega_N\},\\
		&\{\vert v_j^\infty(\hat{x}, \omega)\rvert:\hat{x}\in\mathbb{S}_+^{n-1},\ \omega\in\Omega_N\},
		\quad j = 1,2.
	\end{align*}
\end{problem} 

For simplicity, we denote
$$
	u^\infty=u^\infty(\hat{x}, \omega),\quad
	v_j^\infty=v_j^\infty(\hat{x}, \omega).
$$
Then we derive from \eqref{eq: vj_far} that
\begin{align*}
	v_j^\infty=\Re u^\infty - c_j\Re\Phi_{\omega,j}^\infty + 
	\mathrm{i}\left(\Im u^\infty-c_j\Im\Phi_{\omega,j}^\infty\right),
\end{align*}
with
\begin{align*}
	&\Re\Phi_{\omega,j}^\infty=
	\begin{cases}
		T(\theta)\cos(k_-\hat{x}^t\cdot z_j),&z_j\in\mathbb{R}_-^n,\\
		H(\theta)\cos(k_+\hat{x}\cdot z_j^s)+\cos (k_+\hat{x}\cdot z_j),&z_j\in\mathbb{R}_+^n,
	\end{cases}\\
	&\Im\Phi_{\omega,j}^\infty=
	\begin{cases}
		-T(\theta)\sin(k_-\hat{x}^t\cdot z_j), & z_j\in\mathbb{R}_-^n, \\
		-H(\theta)\sin(k_+\hat{x}\cdot z_j^s)-\sin (k_+\hat{x}\cdot z_j), & z_j\in\mathbb{R}_+^n,
	\end{cases}
\end{align*}
Furthermore, we derive that
\begin{align}\label{eq: u_inf_norm}
	&\lvert u^\infty\rvert^2=\left(\Re u^\infty\right)^2+\left(\Im u^\infty\right)^2,\\
	&\lvert v_j^\infty\rvert^2=\left(\Re u^\infty-c_j\Re\Phi_{\omega,j}^\infty\right)^2+
	\left(\Im u^\infty-c_j\Im\Phi_{\omega,j}^\infty\right)^2.\label{eq: vj_inf_norm}
\end{align}
Subtracting \eqref{eq: u_inf_norm} from \eqref{eq: vj_inf_norm} gives that 
\begin{align}\label{eq: PR}
	\Re\Phi_{\omega,j}^\infty\Re u^\infty+\Im\Phi_{\omega,j}^\infty\Im u^\infty
	=f_j,\quad j=1,2,
\end{align}
with
\[
f_j=-\frac{1}{2c_j}\left(\lvert v_j^\infty\rvert^2-\lvert u^\infty\rvert^2-c_j^2\lvert\Phi_{\omega,j}^\infty\rvert^2\right).
\]

Once \eqref{eq: PR} is solvable, the phase information of $u^\infty$ can be retrieved correspondingly, which further leads to the phased far field:
\begin{align}
	u^\infty=\frac{\det D^R}{\det D}+\mathrm{i}\frac{\det D^I}{\det D},
\end{align}
with 
$$
	D=
	\begin{bmatrix}
		\Re\Phi_{\omega,1}^\infty&\Im\Phi_{\omega,1}^\infty\\
		\Re\Phi_{\omega,2}^\infty&\Im\Phi_{\omega,2}^\infty
	\end{bmatrix},\quad
	D^R=
	\begin{bmatrix}
		f_1&\Im\Phi_{\omega,1}^\infty\\
		f_2&\Im\Phi_{\omega,2}^\infty
	\end{bmatrix},\quad
	D^I=
	\begin{bmatrix}
		\Re\Phi_{\omega,1}^\infty&f_1\\
		\Re\Phi_{\omega,2}^\infty&f_2
	\end{bmatrix}.
$$
To ensure the unique solvability of \eqref{eq: PR}, the determination of $D$ should be non-zero, which can be fulfilled by choosing $\alpha_j,\,j=1,2,$ properly. Instead of heading for the mathematical argument to clarify the stability of the phase retrieval technique, we shall illustrate the performance of the proposed through several numerical examples. Nevertheless, the mathematical analysis to prove the stability of phase retrieval can be obtained similarly to that in \cite{Wang2023}.

By solving \eqref{eq: PR}, we can obtain the phase information of the radiated field. We can then employ the Fourier method developed in \Cref{sec: Fourier} to reconstruct the source function.


\section{Numerical examples}\label{sec: example}

In this section, we shall conduct several numerical experiments to verify the algorithms proposed in this paper. Both two- and three-dimensional cases are considered. 
The forward problem is solved by direct integration. The admissible wavenumber set in the lower half-space $\mathbb{R}_-^n $ is given by
$$
	\mathbb{K}_{-, N}=\left\{k: k=2\pi\lvert\vec{l}\rvert,\ \vec{l}\in\mathbb{L}_N\right\}\cup\left\{2\pi\lambda\right\},\ \lambda=10^{-3},
$$
with
$$
	\mathbb{L}_N=\left\{\vec{l}:\vec{l}=(l_1,l_2,\cdots,l_n)\in\mathbb{Z}^n,\ 1\le\lvert\vec{l}\rvert_\infty\le N,l_n>0,\ \theta_{\vec{l}}\in\Theta_c^n\right\},
$$
and $N$ is chosen to be 50 for convenience. Once $\mathbb{K}_{-,N}$ is chosen properly, the admissible wave frequency $\Omega_N$ and the admissible wave number $\mathbb{K}_{+,N}$ in $\mathbb{R}_+^n$ are set to be
\begin{align*}
	\Omega_N=\left\{\omega:\omega=k_-c_-,\ k_-\in\mathbb{K}_{-, N}\right\},\quad
	\mathbb{K}_{+,N}=\Omega_N/c_+.
\end{align*}

In the following examples, the wave speed $c_\pm$ is taken to be $c_-=2,$ $c_+=c_--\pi/1000.$ With the admissible wave number set chosen above, we obtain the synthesis far-field data from \eqref{eq: uinf} for $\text{supp} f\subset V_0$. We use the Gauss quadrature to calculate the volume integrals over the $100\times 100$ or $50^3$ Gauss-Legendre points for $V_0\subset\mathbb{R}^n,\,n=2,3.$ Specifically, $V_0$ is chosen to be $[-0.5,0.5]^{n-1}\times[-0.5,0]$.

To better exhibit the reconstruction, we divide this section into two subsections. In \Cref{sub1}, we show the performance of the phase retrieval technique, both in two- and three-dimensional space. In \Cref{sub2}, we shall verify the effectiveness of the Fourier method.
The exact source functions are chosen to be 
\begin{align*}
S_{\rm 2D}(x)&=1.1\exp\left(-200((x_1-0.01)^2+(x_2+0.38)^2)\right)\\
&\quad-100\left((x_2+0.5)^2-x_1^2\right)\exp\left(-90\left(x_1^2+(x_2+0.5)^2\right)\right)
\end{align*}
for $n=2,$ and
\begin{align*}
	S_{\rm 3D}(x)&=1.1\exp\left(-200((x_1-0.01)^2+(x_2-0.12)^2+(x_3+0.5)^2)\right)\\
	&\quad -100\left(x_2^2-x_1^2\right)\exp\left(-90\left(x_1^2+x_2^2+(x_3+0.5)^2\right)\right)
\end{align*}
for $n=3.$


\subsection{Phase retrieval}\label{sub1}
In this subsection, we shall test the performance of the phase retrieval technique developed in \Cref{sec: PR}.
We consider two cases: In the first case, all the artificial source points are located in $\mathbb{R}_-^n,$ while in the second case, all the artificial source points are located in $\mathbb{R}_+^n.$
As described before, the choice of the parameter $\alpha_j,\,j=1,2$ and the reference source points $z_j,\,j=1,2$ is of vital significance. In our experiment, we take $\alpha_1$ to be 
$\alpha_1=1/2.$ For $k_-\in\mathbb{K}_{-,N}\backslash\left\{k_{-,\vec{0}}\right\},$ we take $\alpha_2=1/2$ and $-4$ for $k=k_{-,\vec{0}}.$

Further, to illustrate the stability of the numerical method, the noisy phaseless far-field data $u^{\infty,\epsilon}$ takes the form
$$
	\lvert u^{\infty,\epsilon}\rvert:=(1+\epsilon  r)\lvert u^\infty\rvert,
$$
where $r\in[-1,1]$ is a uniform random number, $\epsilon>0$ is the noise level, and 
\[
\{u^\infty=u^\infty(\hat{x}_j, \omega_j):\omega_j\in\Omega_N,\ j=1,2,\cdots,(2N+1)^n\}.
\]

To evaluate the accuracy of the phase retrieval technique quantitatively, we compute the relative $L^2$- and $L^\infty$-errors between the exact far-field data and the retrieval data for $n=2$, which are respectively defined by:
\begin{align*}
	\text{Err}_{L^2}&=\frac{\left(\sum\limits_{\stackrel{1\le\lvert\vec{l}\rvert_\infty\le N}{\theta_{\vec{l}}\,\in\Theta_c^2}}\lvert u^\infty(\hat{x}_{\vec{l}}, \omega_{\vec{l}})-u^{\infty,\epsilon}(\hat{x}_{\vec{l}}, \omega_{\vec{l}})\rvert^2\right)^{1/2}}
	{\left(\sum\limits_{\stackrel{1\le\lvert\vec{l}\rvert_\infty\le N}{\theta_{\vec{l}}\in\Theta_c^2}}\lvert u^\infty(\hat{x}_{\vec{l}}, \omega_{\vec{l}})\rvert^2\right)^{1/2}},\\
	\text{Err}_\infty & =\frac{\max\limits_{\stackrel{1\le\lvert\vec{l}\rvert_\infty\le N}{\theta_{\vec{l}}\,\in\Theta_c^2}}\lvert u^\infty(\hat{x}_{\vec{l}}, \omega_{\vec{l}})-u^{\infty,\epsilon}(\hat{x}_{\vec{l}}, \omega_{\vec{l}})\rvert}{\max\limits_{\stackrel{1\le\lvert\vec{l}\rvert_\infty\le N}{\theta_{\vec{l}}\,\in\Theta_c^2}}\lvert u^\infty(\hat{x}_{\vec{l}}, \omega_{\vec{l}})\rvert}.
\end{align*}

By taking different noise levels $\epsilon, $ we compute the relative errors. In \Cref{tab: error2D}--\Cref{tab: error2U}, we respectively exhibit the relative errors where the reference points are located in $\mathbb{R}_-^2,\,\mathbb{R}_+^2$, which illustrates that the phase retrieval technique developed here exhibits a superior ability in recovering the phase information, no matter whether the reference source points are located above or below the flat plane $\Gamma_0$. Especially for the case where no noise is involved in the observed data, the $L^2/L^\infty-$ error of the phase retrieval is surprisingly low to the machine precision. Even under noisy data, the error of phase retrieval can be controlled less than the noise level. We want to point out that the relative errors computed here are the sum of the relative errors for all frequencies. Specifically, we take $N=50$ here, illustrating that the errors computed in \Cref{tab: error2D}--\Cref{tab: error2U} is the sum of the errors corresponding to 5065 wave numbers.

\begin{table}[htp]
  \centering
	\caption{The relative errors $\text{Err}_{L^2}$ and $\text{Err}_\infty$ with the reference source points located in $\mathbb{R}_-^2$.}\label{tab: error2D}
	\begin{tabular}{ccccccc}
		\toprule
		$\epsilon$&$0$&$0.5\%$&$1\%$&$2\%$&$5\%$&$10\%$\\
		\midrule
		$\text{Err}_{L^2}$&$1.69\times10^{-16}$&$0.32\%$&$0.68\%$&$1.44\%$&$3.71\%$&$7.17\%$\\
		$\text{Err}_\infty$&$4.53\times10^{-16}$&$0.39\%$&$0.84\%$&$2.13\%$&$6.11\%$&$13.34\%$\\
		\bottomrule
	\end{tabular}
\end{table}

\begin{table}[htp]
   \centering
	\caption{The relative errors $\text{Err}_{L^2}$ and $\text{Err}_\infty$ with the reference source points located in $\mathbb{R}_+^2$.}\label{tab: error2U}
	\begin{tabular}{ccccccc}
		\toprule
		$\epsilon$&$0$&$0.5\%$&$1\%$&$2\%$&$5\%$&$10\%$\\
		\midrule
		$\text{Err}_{L^2}$&$3.07\times10^{-16}$&$0.36\%$&$0.78\%$&$1.57\%$&$4.28\%$&$6.81\%$\\
		$\text{Err}_\infty$&$4.81\times10^{-16}$&$0.58\%$&$1.11\%$&$1.92\%$&$5.41\%$&$11.66\%$\\
		\bottomrule
	\end{tabular}
\end{table}

Next, we compute the reconstruction errors of the phase information for $n=3.$ To reduce the computational cost, we define the error corresponding to Fourier index ${\vec{l}}$ as follows:
\begin{align*}
	\text{Err}(\omega_{\vec{l}})&=\frac{\lvert u^\infty(\hat{x}_{\vec{l}}, \omega_{\vec{l}})-u^{\infty,\epsilon}(\hat{x}_{\vec{l}}, \omega_{\vec{l}})\rvert}
	{|u^\infty(\hat{x}_{\vec{l}}, \omega_{\vec{l}})|},
\end{align*}
and further exhibit the error subject to different $\vec{l}=(l_1, l_2, l_3)$ in \Cref{tab: error3D}--\Cref{tab: error3U}.
All the numerical results for $n=2$ and $n=3$ indicate that our phase retrieval technique recovers the phase information with high precision. Taking the ill-posedness of the inverse source problem into account, it can be seen that the error caused by the phase retrieval has an almost negligible effect on the final reconstruction. 

\begin{table}[htp]\centering
\caption{The relative errors $\text{Err}(\omega_{\vec{l}})$ with the reference source points located in $\mathbb{R}_-^3$.}\label{tab: error3D}
\begin{tabular}{ccccccc}
	\toprule
	 &$0$&$0.5\%$&$1\%$&$2\%$&$5\%$&$10\%$\\
	\midrule
	$(-2,0,1)$&$2.00\times10^{-16}$&$0.46\%$&$0.85\%$&$1.36\%$&$5.84\%$&$8.10\%$\\
	$(1,0,3)$&$6.25\times10^{-17}$&$0.34\%$&$0.98\%$&$0.34\%$&$3.68\%$&$4.32\%$\\
	$(17,-13,0)$&$8.58\times10^{-17}$&$0.70\%$&$1.29\%$&$2.83\%$&$2.27\%$&$7.84\%$\\
	$(-27,9,14)$&$1.35\times10^{-15}$&$0.42\%$&$0.58\%$&$1.82\%$&$1.55\%$&$6.78\%$\\
	$(-30,-10,23)$&$4.47\times10^{-16}$&$0.49\%$&$0.72\%$&$1.97\%$&$4.74\%$&$6.19\%$\\
	\bottomrule
\end{tabular}
\end{table}

\begin{table}[htp]
\centering
\caption{The relative errors $\text{Err}(\omega_{\vec{l}})$ with the reference source points located in $\mathbb{R}_+^3$.}\label{tab: error3U}
\begin{tabular}{ccccccc}
	\toprule
	&$0$&$0.5\%$&$1\%$&$2\%$&$5\%$&$10\%$\\
	\midrule
	$(-2,0,1)$&$2.00\times10^{-16}$&$0.35\%$&$0.96\%$&$1.23\%$&$6.67\%$&$8.05\%$\\
	$(1,0,3)$&$6.24\times10^{-17}$&$0.46\%$&$0.82\%$&$1.31\%$&$1.96\%$&$3.79\%$\\
	$(17,-13,0)$&$8.58\times10^{-17}$&$0.11\%$&$0.41\%$&$1.83\%$&$5.04\%$&$8.60\%$\\
	$(-27,9,14)$&$1.35\times10^{-15}$&$0.18\%$&$0.32\%$&$1.55\%$&$1.64\%$&$3.25\%$\\
	$(-30,-10, 23)$&$4.47\times10^{-16}$&$0.26\%$&$0.52\%$&$2.52\%$&$2.09\%$&$7.99\%$\\
	\bottomrule
\end{tabular}
\end{table}

\subsection{Source reconstruction}\label{sub2}
In this subsection, we focus on recovering the source from the recovered phased radiated data by utilizing the Fourier method proposed in \Cref{sec: Fourier}. 
In \Cref{fig: 2D}, we plot the exact source functions $S_{\rm 2D}$ and its reconstruction. We can see from \Cref{fig: 2D} that the rough contour of the source function can be captured, while the value may not be matched well. A reason accounting for this is that the observation aperture is limited (less than $\pi$), and there is no sufficient information available to produce a perfect reconstruction. As expected, the reconstruction can be improved largely if we only extend the domain of definition for transmitted direction by removing the restriction that $\theta_{\vec{l}}\in\Theta_c^n$ from \eqref{eq: XN}. We refer to \Cref{fig: 2D}(c) for this case.  

Collecting the above observation, we find that the Fourier method is capable of recovering the source function $S_{\rm 2D}$ from the phaseless radiated field. Under the two-layered medium background, where the observation aperture is limited, the accuracy of the reconstruction should be considered to be acceptable. 

\begin{figure}[htpb]
	\centering
	\subfigure[]{\includegraphics[width=0.3\linewidth]{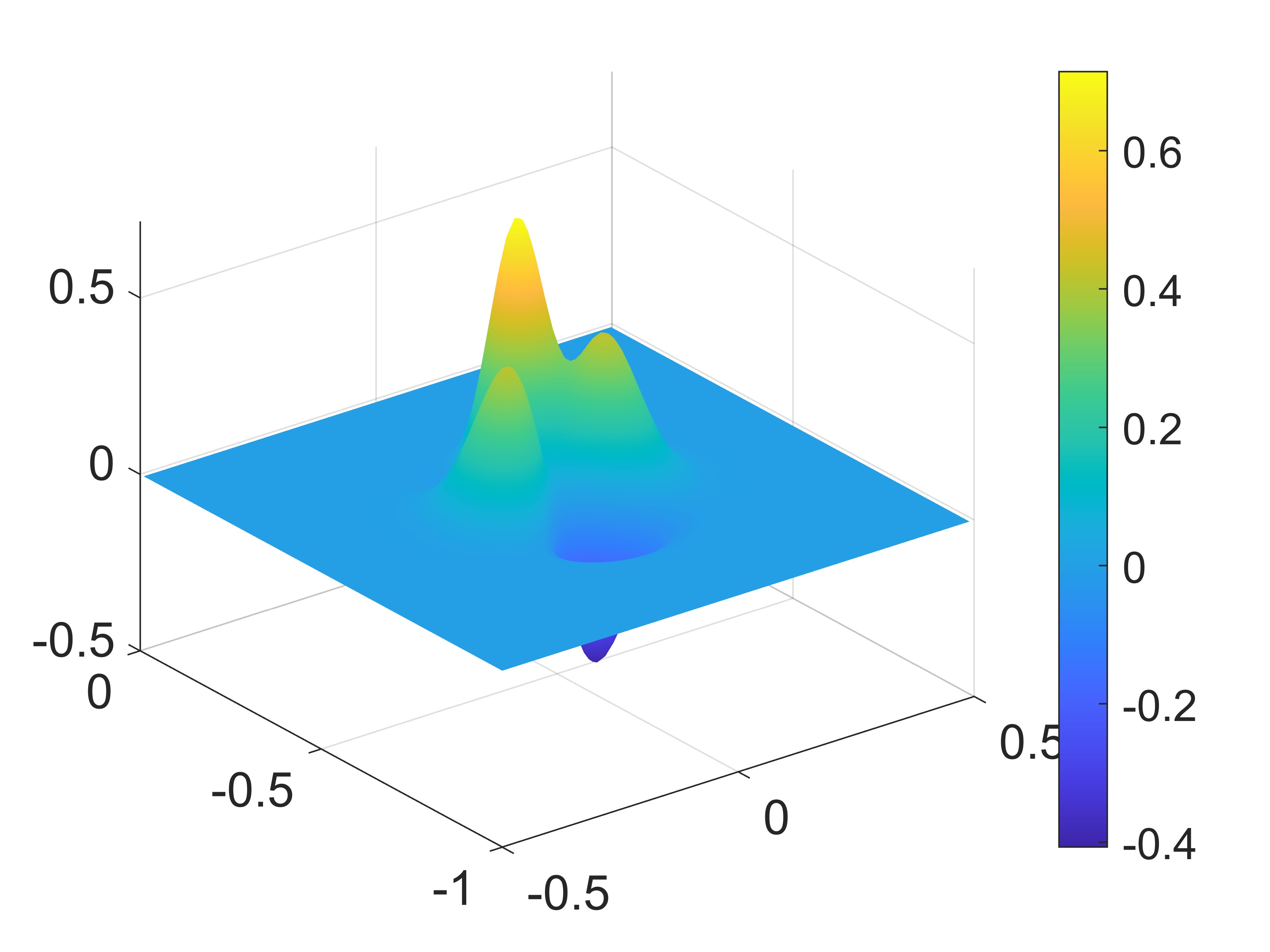}}
	\subfigure[]{\includegraphics[width=0.3\linewidth]{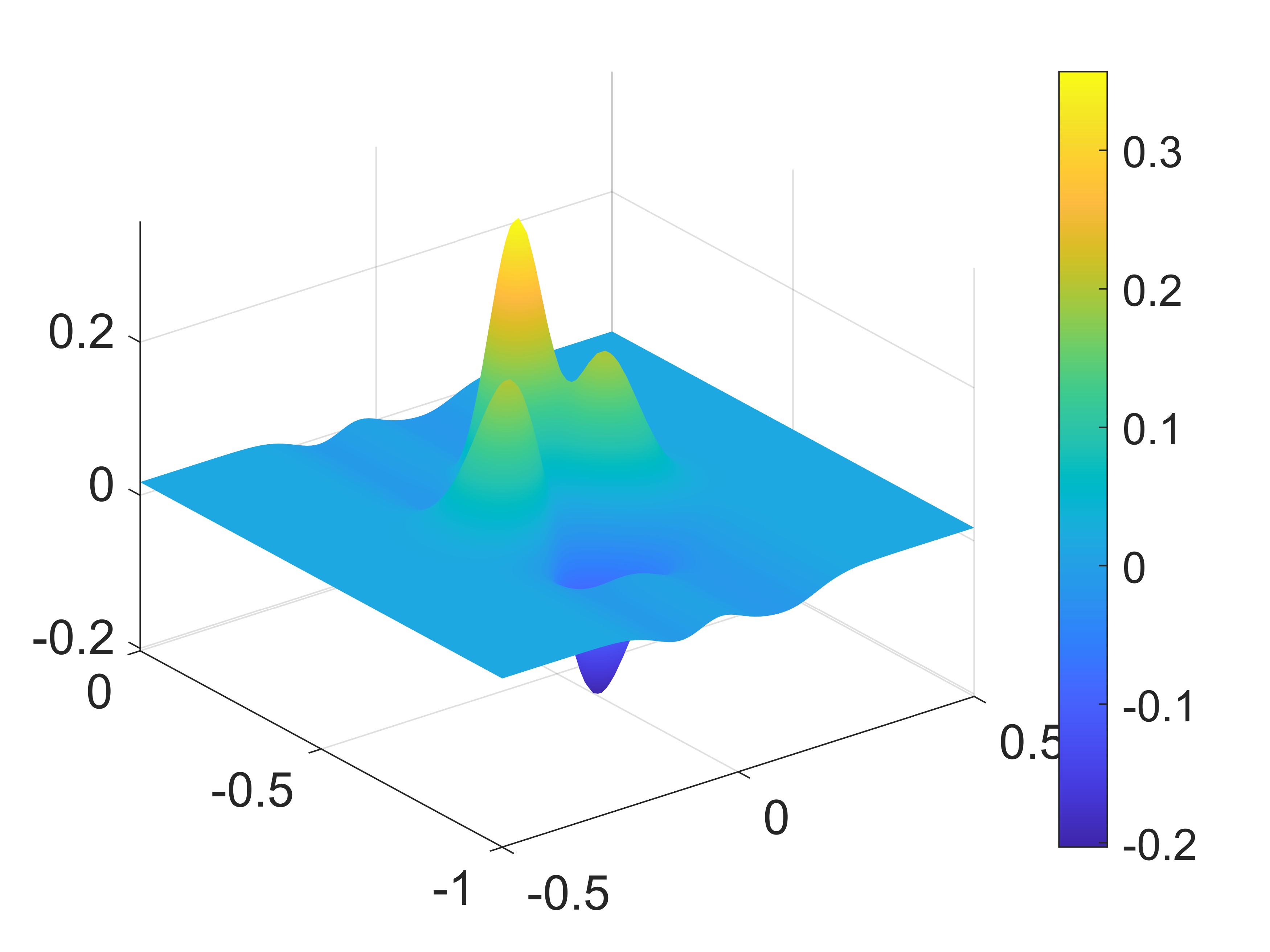}}
	\subfigure[]{\includegraphics[width=0.3\linewidth]{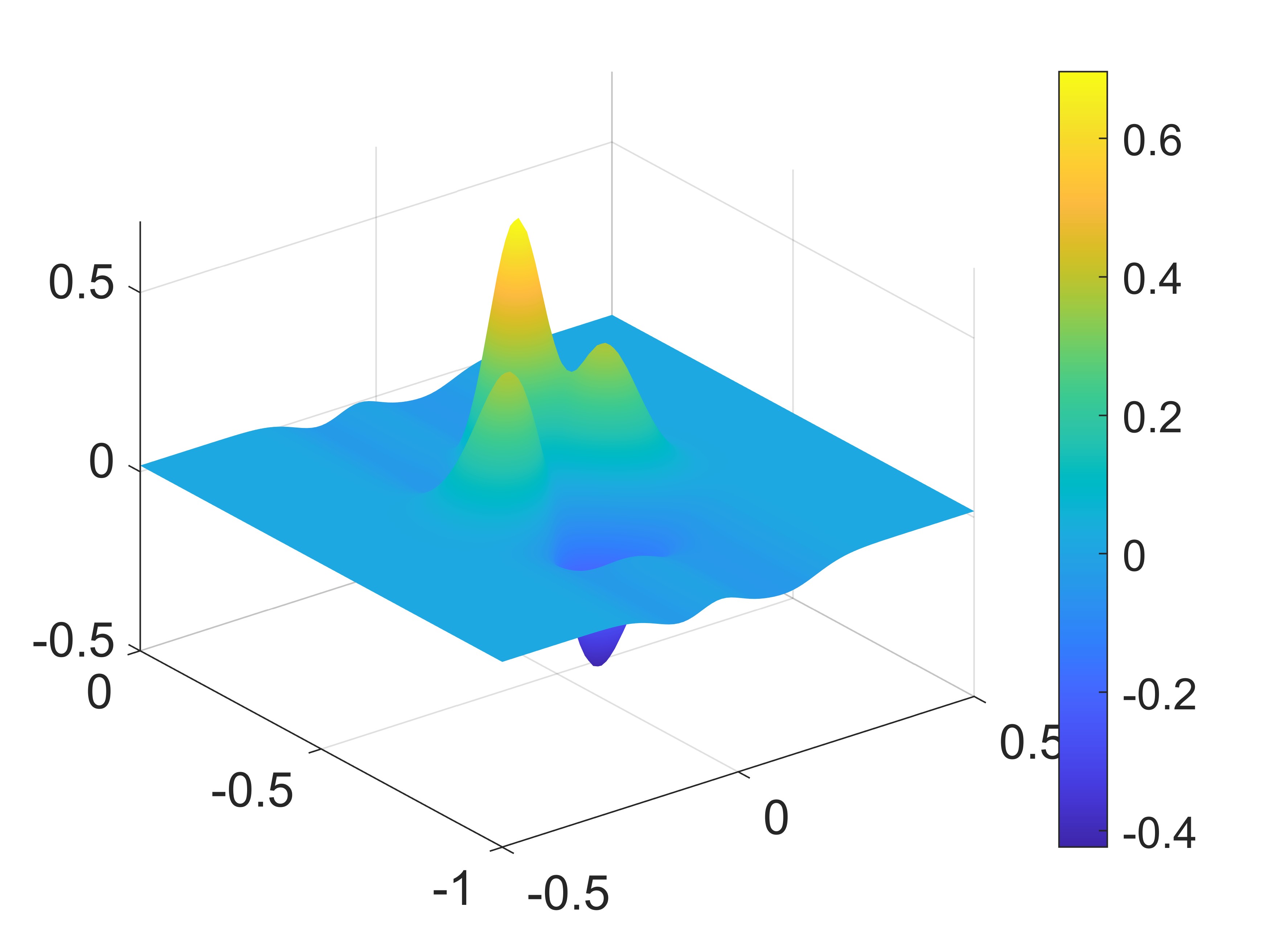}}
	\caption{The exact source function and its reconstruction. (a) Surface plot (b) Reconstruction (c) Reconstruction by removing the restriction that $\theta_{\vec{l}}\in\Theta_c^n$ from \eqref{eq: XN}}\label{fig: 2D}
\end{figure}

As the last numerical experiment, we consider the reconstruction of the source function $S_{\rm 3D}(x)$ from the recovered phased radiated field.
In \Cref{fig: 3D}, we plot the exact and the reconstructed source functions at three cross-sections: $x_1=0.01,\,x_2=0.12,\,x_3=-0.5$, respectively.
One can see from \Cref{fig: 3D} that our method can roughly approximate the source function while the accuracy seems to be not so satisfactory, which is a result of the fact that the observation data is only available on the limited aperture.

\begin{figure}[htpb]
	\centering
	\subfigure[]{\includegraphics[width=0.45\linewidth]{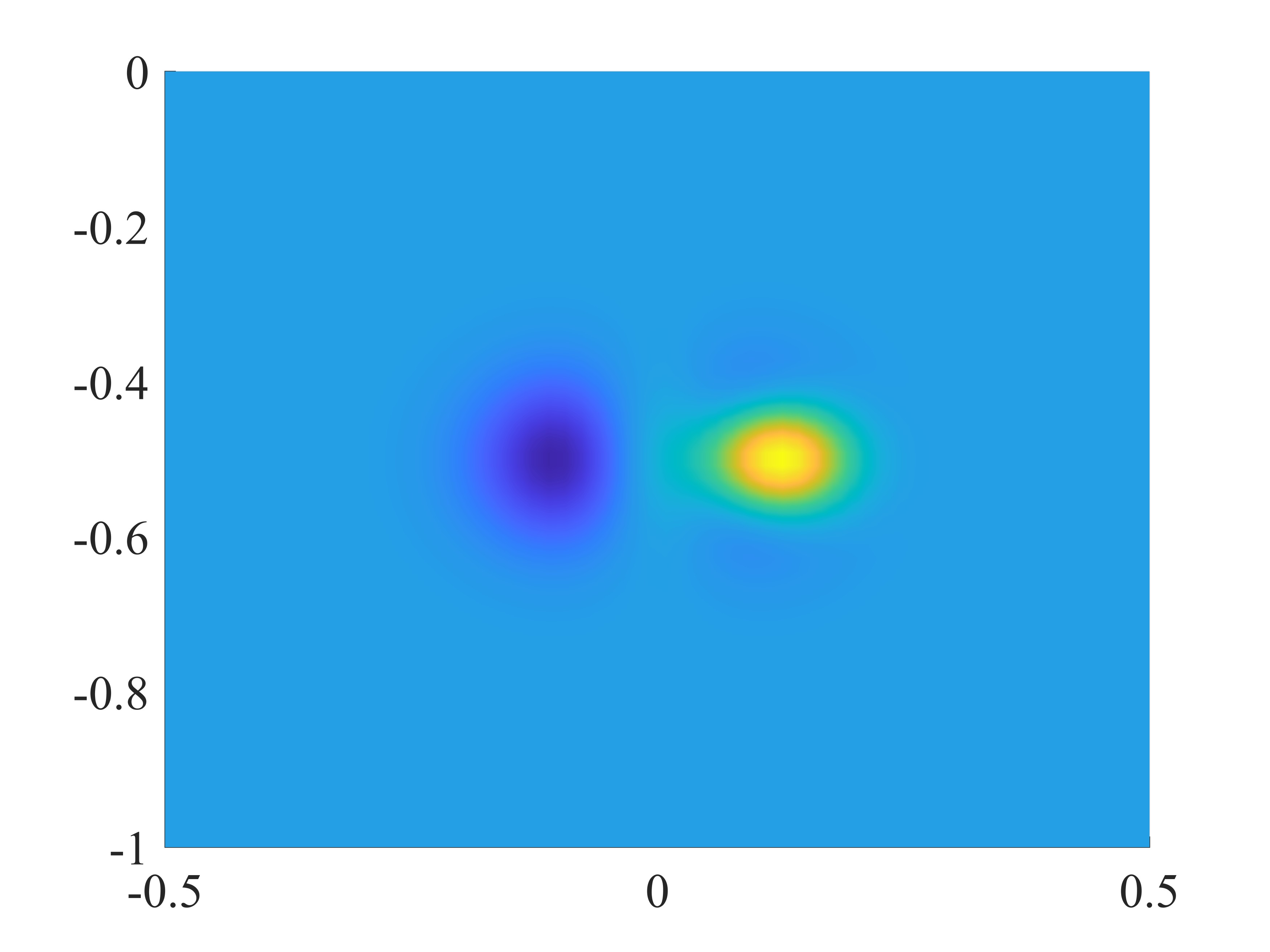}}
	\subfigure[]{\includegraphics[width=0.45\linewidth]{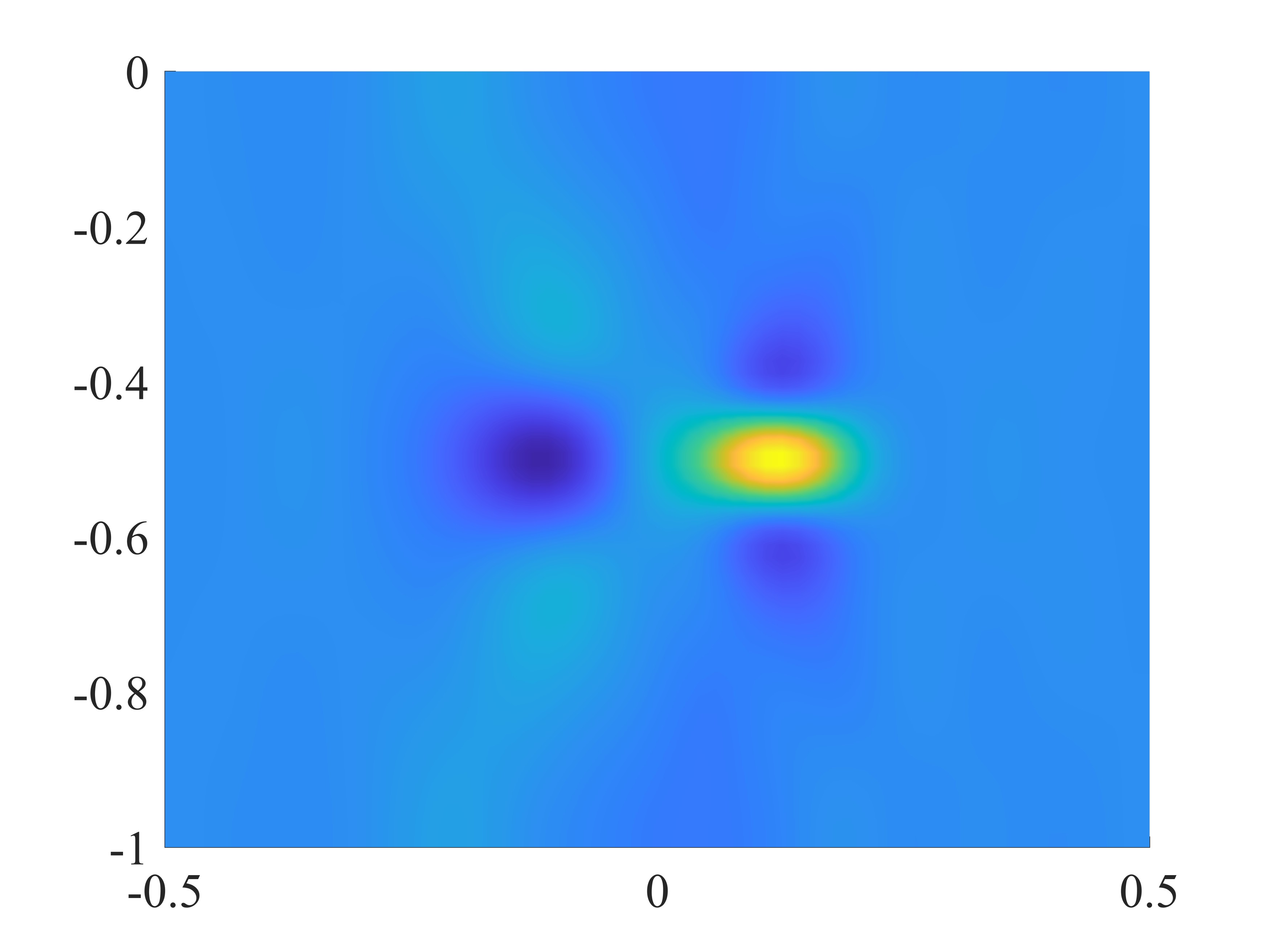}}\\
	\subfigure[]{\includegraphics[width=0.45\linewidth]{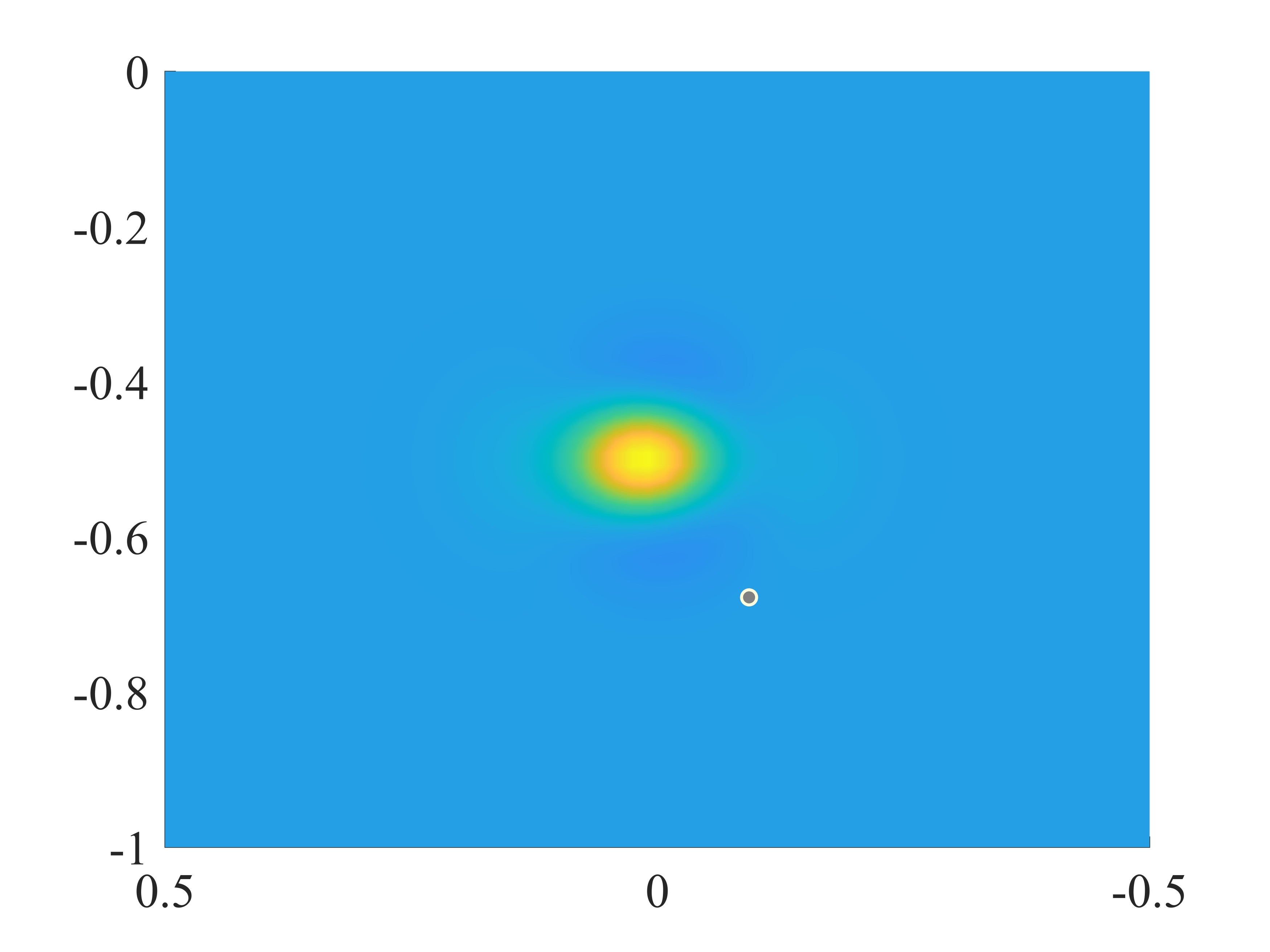}}
	\subfigure[]{\includegraphics[width=0.45\linewidth]{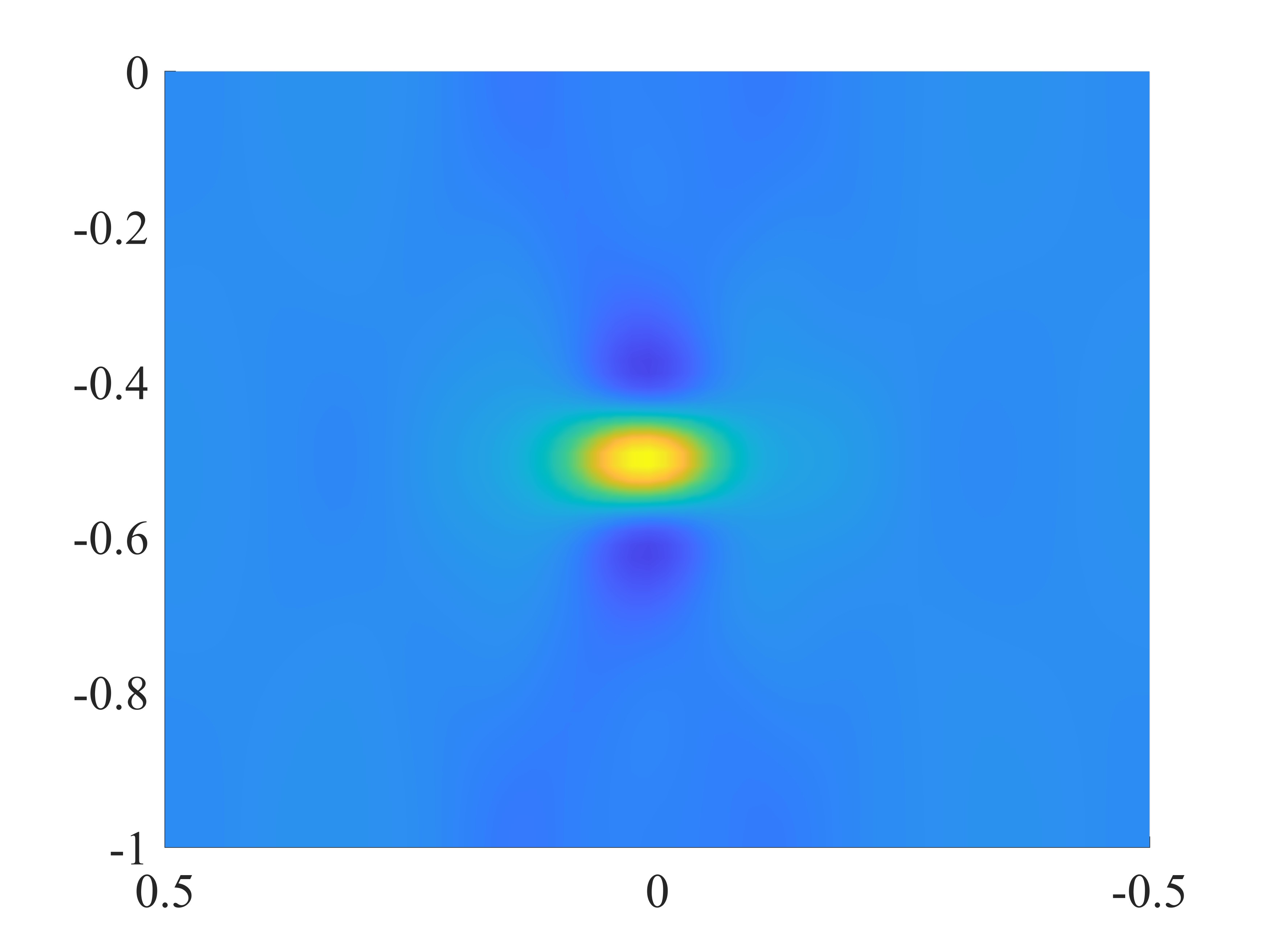}}\\
	\subfigure[]{\includegraphics[width=0.45\linewidth]{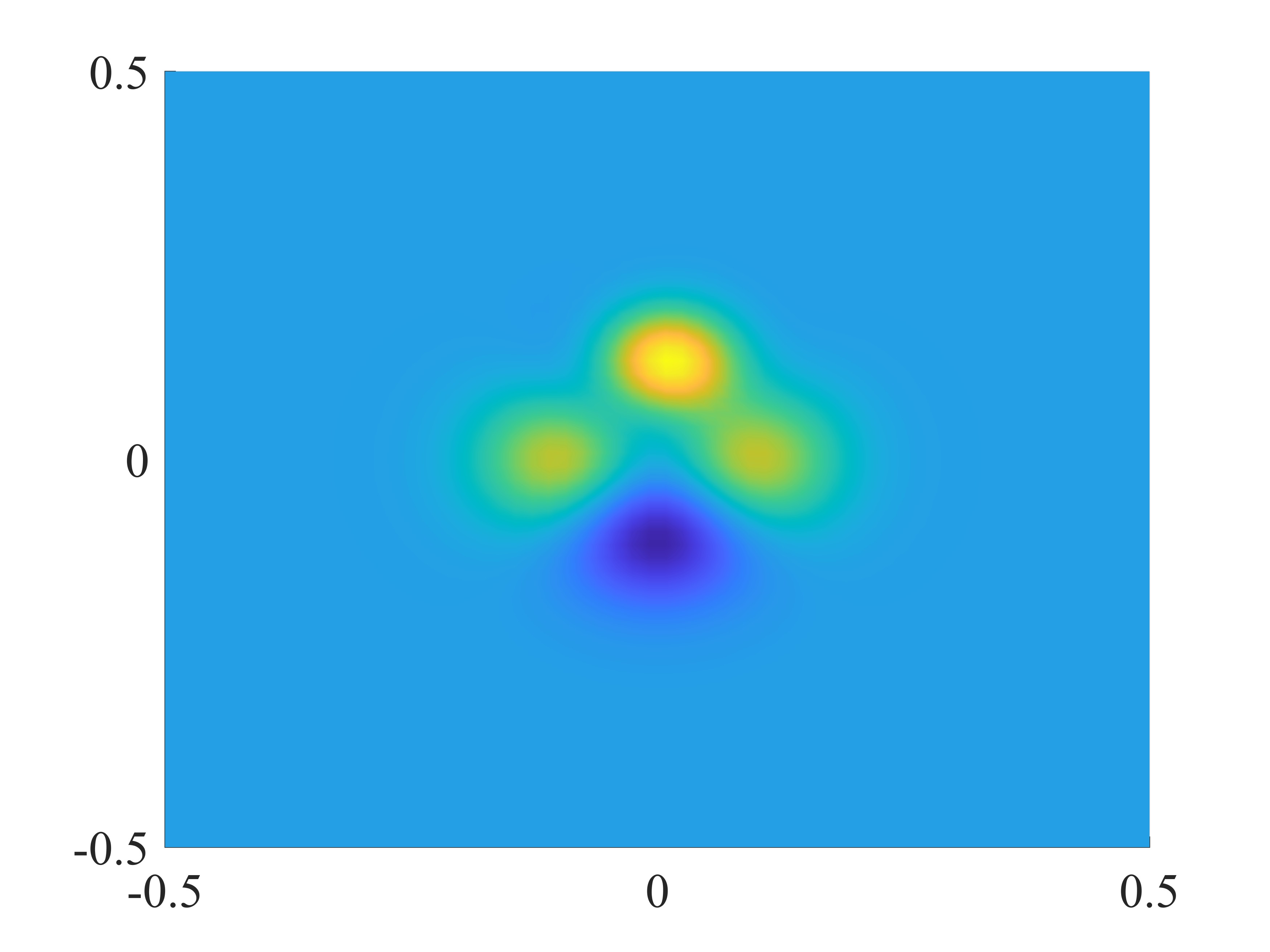}}
	\subfigure[]{\includegraphics[width=0.45\linewidth]{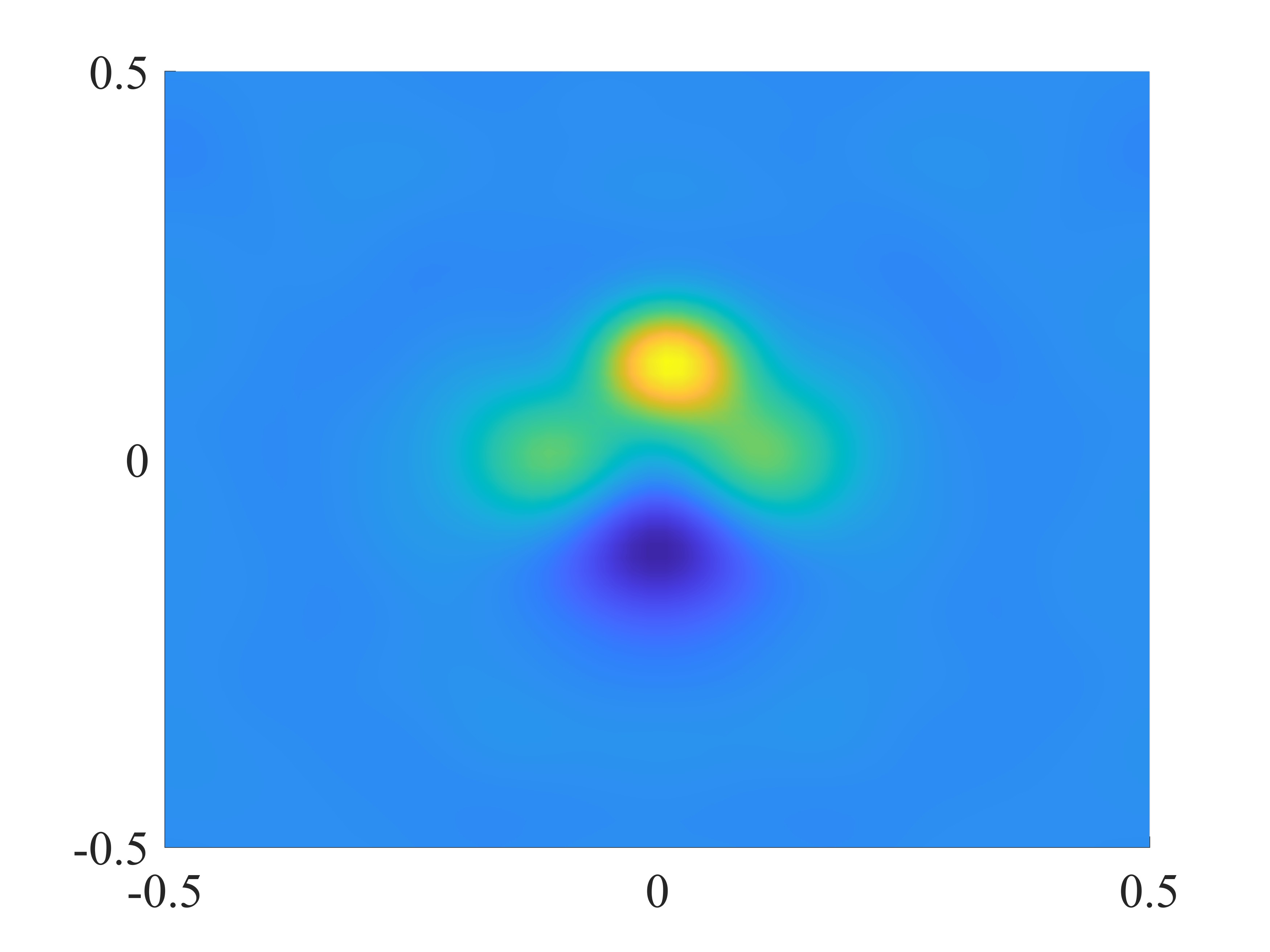}}
	\caption{The exact source function and its reconstruction in 3D. Column 1: exact sources; Column 2: reconstructions; Row 1: $x_1=0.01,$ Row 2: $x_2=0.12,$ Row 3: $x_3=-0.5.$}\label{fig: 3D}
\end{figure}

At the end of this section, we want to remark that, the retrieval technique developed can well recover the phase information from the phaseless far-field data. The Fourier method performs well in identifying the source function to some extent. Nevertheless, the accuracy seems to be a bit low due to the limited observation aperture, which is a challenge due to the physical model under consideration.

\section{Conclusion}\label{sec: conclusion}

In this paper, we develop a numerical method based on the Fourier expansion to determine the source function in the two-layered medium from multi-frequency far-field patterns. For the phased measured data, we extend the Fourier method developed in \cite{Wang2017} to the inverse source problem in the two-layered medium problem. For the phaseless data, we develop a phase retrieval technique by introducing several auxiliary source points to the inversion model. Numerical examples in two- and three-dimensions are conducted to verify the performance of the proposed method. The phase information can be well recovered by the phase retrieval technique and the Fourier method can accurately reconstruct the source function and no forward solver is involved.

Improvements in the experimental results can be further investigated, especially the quality enhancement under limited observation aperture. In addition, it would be interesting to investigate the case where the background medium has more than two layers. We hope to report developments in these directions in future works.





\bibliographystyle{plain}
\bibliography{references}
\end{document}